\documentclass[12pt]{amsart}
\usepackage[dvipdfmx]{graphicx}
\title[Weinstein trisections of trivial surface bundles.]
{Weinstein trisections of trivial surface bundles.}
\author{Masaki Ogawa}
\address{Mathematical science center for co-creative society, Tohoku University, Aoba-6-3 Aramaki, Aoba Ward, Sendai, Miyagi 980-0845}
\email{masaki.ogawa.b7@tohoku.ac.jp}

\usepackage[utf8]{inputenc}
\usepackage{amsmath,amssymb,amsthm}
\usepackage{epsfig}
\usepackage{txfonts}
\usepackage{graphicx}
\usepackage{ulem}
\usepackage{color}
\usepackage[all]{xy}
\usepackage{fancybox}
\usepackage{lscape}
\usepackage{comment}
\usepackage{caption}
\usepackage{here}
\usepackage{setspace}
\usepackage{bm}

\theoremstyle{plain}
\newtheorem{theorem}{Theorem}[section]
\newtheorem{lemma}[theorem]{Lemma}
\newtheorem{corollary}[theorem]{Corollary}

\newtheorem{question}[theorem]{Question}

\theoremstyle{definition}
\newtheorem{definition}[theorem]{Definition}

\theoremstyle{definition}

\definecolor{myred}{rgb}{.8,.0,.0}
\definecolor{mygreen}{rgb}{.0,.6,.0}
\definecolor{mygray}{gray}{0.7}

\makeatletter 
 \@mparswitchfalse 
\makeatother

\newcounter{mystepcount}

\begin{document}
\maketitle
\begin{abstract}
	Weinstein trisection is a trisection of a symplectic 4-manifold whose 1-handlebodies are the Weinstein domain for the symplectic structure induced from an ambient manifold.
	Lambert-Cole, Meier, and Starkston showed that every closed symplectic 4-manifold admits a Weinstein trisection. In this paper, we construct a Weinstein trisection of $\Sigma_g\times \Sigma_h$. 
	As a consequence of this construction, we construct a little explicit Weinstein trisection of $S^2\times S^2$.
\end{abstract}
\section{introduction}
A trisection, introduced in \cite{GK} has been studied by many authors. 
One of the kind of the field in trisection theory is a Weinstein trisection.
Lambert-Cole reproved the Thom conjecture by using trisection of $\mathbb{C}P^2$ and inequality about an invariant induced by Khovanov homology.
In the proof, he used a contact structure induced in the spine of trisection.

Recently, Lambert-Cole, Meier, and Starkston introduce a trisection adapted to a symplectic 4-manifold called a Weinstein trisection\cite{LMS}.
They also showed that every closed symplectic 4-manifold admits a Weinstein trisection.
Weinstein trisection sometimes gives us a tool to study symplectic closed 4-manifolds and symplectic surfaces in them \cite{L1, L2}.

There are some examples of Weinstein trisection. 
For example, Weinstein trisection of $\mathbb{C}P^2$ and $\mathbb{C}P^2\#  \overline{\mathbb{C}P^2}$ is described in \cite{LMS}.
On the other hand, given a finite group $G$, there is a symplectic closed 4-manifold with fundamental group isomorphic $G$.
This means that there are so many symplectic closed 4-manifolds. 
Therefore, it is very important to make an example of Weinstein trisection as the first step to studying symplectic closed 4-manifolds with Weinstein trisection.
The main result of this paper is the following:
\begin{theorem}\label{thm1}
	$\Sigma_g\times \Sigma_h$ admits a genus $(2g+1)(2h+1)+1$ Weinstein trisection for some symplectic structure.
\end{theorem}

This implies that the trisection genus of $\Sigma_g\times \Sigma_h$ equals the Weinstein trisection genus of its.

\begin{corollary}
	The trisection genus of $\Sigma_g\times \Sigma_h$ equals the Weinstein trisection genus of its.
\end{corollary}

This corollary immediately follows from the result in \cite{W2}. It states that the trisection genus of $\Sigma_g\times \Sigma_h$ is $(2g+1)(2h+1)+1$.

This paper is organized as follows:
In Section 2, we review definitions of trisection and Weinstein trisection. After that, we introduce a trisection of $\Sigma_g\times \Sigma_h$ which is constructed in \cite{W2} in Section 3.
Then, we show this trisection can be seen as a Weinstein trisection in Section 4.

\section{preliminalies}
In this section, we set up the notion and objects we use in this paper.
Trisection of 4-manifolds is a decomposition of a four-manifold.
 \begin{definition}
Let $g,k_1,k_2$ and $k_3$ be non-negative integers with $\max\{k_1,k_2,k_3\}\leq g$. 
A $(g; k_1, k_2, k_3)$-{\it trisection} of a closed 4-manifold $X$ 
is a decomposition $X=X_1\cup X_2\cup X_3$ such that for $i,j\in\{1,2,3\}$, 
\begin{itemize}
 \item
$X_i\cong \natural^{k_i}(S^1\times B^3)$, 
 \item
$H_{ij}=X_i\cap X_j\cong \natural^g (S^1\times B^2)$ if $i\neq j$, and 
 \item
$\Sigma=X_1\cap X_2\cap X_3 \cong \#^g (S^1\times S^1)$. 
\end{itemize}
\end{definition}

Symplectic 4-manifolds is a 4-manifold with non-degenerate closed 2-form $\omega$. We denote it by $(X, \omega)$.
In the field of symplectic topology, symplectic manifolds are considered the same if they a symplectomorphic.
\begin{definition}
	Let $(X, \omega)$ be a symplectic manifold and $\varphi$ a diffeomorphism between itself.
	We say $\varphi$ is a symplectomorphism if it preserves the symplectic form $\omega$ (i.e. $\varphi^\ast \omega=\omega$).
\end{definition}
Since $\omega$ is a non-degenerate, we can obtain a volume form $\Omega=\omega\wedge\cdots \wedge \omega$ of $X$ after taking a wedge product n times.
Hence in dimension two, the volume form of it is also a symplectic form so, volume-preserving diffeomorphism is the same as a symplectomorphism.

Lambert-Cole, Meier, and Starkston define a trisection of a symplectic closed 4-manifold adopted to the symplectic structure.
The Weinstein domain, first introduced in \cite{W}, is a symplectic manifold with a contact boundary.
More precisely, let $(X, \omega)$ be a compact symplectic manifold with boundaries.
Then, $(X, \omega)$ is called a Weinstein domain if there exists a Morse function $f$ on $X$ and a  gradient-like vector field $X_f$ of $f$ such that $X_f$ is Liouville vector field (i.e. $d(\iota_{X_f}\omega)=\omega$).
\begin{definition}
	Let $(X, \omega)$ be a symplectic closed 4-manifold and $X=X_1\cup X_2\cup X_3$ a trisection of $X$.
	We say $X=X_1\cup X_2\cup X_3$  is a Weinstein trisection if there is a Morse function $f_i: X \rightarrow \mathbb{R}$ and gradient-like vector field $X_f$ of $f$ such that $(X_i, \omega|_{X_i}, f_i, X_f)$ is a Weinstein domain, 
\end{definition}

Lambert-Cole, Meier, and Starkston showed that any symplectic closed 4-manifold admits a Weinstein trisection by using a branched covering \cite{LMS}.
Weinstein domains induce a contact structure in their boundaries by a 1-form $\iota_{X_{f}} \omega$. 
Also,  $H_{ij}$ has contact structures induced by a Weinstein structure of $X_i$ and $X_j$ \cite{L2}.

And they give the following question.
 \begin{question}
 	Is a trisection genus equal to the Weinstein trisection genus?
 \end{question}
 
 In this paper, we answer this question for the $\Sigma_g\times \Sigma_h$. 
The known trisections of symplectic four-dimensional manifolds are known as Weinstein trisections.
To positively affirm that there are no trisections that are not Weinstein trisections, we still have too few examples. 

\section{trisection of trivial surface bundle}

In this section we review the construction of a trisection of $\Sigma_g\times \Sigma_h$ in \cite{W2}.
Let $X=\Sigma_g\times \Sigma_h$.
First, we consider the decomposition of $\Sigma_g$.
We can decompose any closed surface into three disks.
\begin{lemma}[Lemma 3.1 in \cite{W2}] \label{decolem1}
	$\Sigma_g$ admits a decomposition $\Sigma_g=B_1\cup B_2\cup B_3$ satisfies the following:
		\begin{enumerate}
			\item Each $B_i$ are disks, 
			\item $b_{ij}=B_i\cap B_j$ is $2g+1$ disjoint arcs.
			\item $B_1\cap B_2\cap B_3$ is $4g+2$ disjoint vertices.
		\end{enumerate}
\end{lemma}

If $g=1$, this decomposition is illustrated in Figure \ref{surfdeco}.  To obtain the case of a higher genus, remove a neighborhood of a vertex and glue a copy of it to itself along their boundary. 
\begin{figure}[h]
\centering
\includegraphics[scale=1.2]{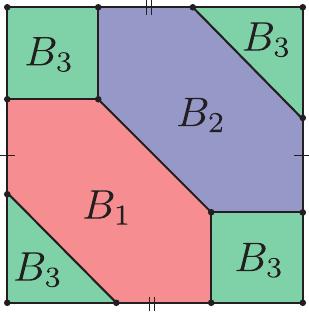}
\caption{The decomposition of $T^2$ in Lemma \ref{decolem1}.}
\label{surfdeco}
\end{figure}

We take mutually disjoint disks $N_i$ in $\Sigma_h$ for $i=1, 2, 3$.
Then we define the 1-handlebodies of a trisection as follows: 
\[
		X_i = (B_i \times (\Sigma_h - Int(N_i ; \Sigma_h)) \cup (B_{i+1}\times N_{i+1}).
\]

To check $X_i$ is a 1-handlebody, we show that $(B_i \times (\Sigma_h - Int(N_i ; \Sigma_h))$ is a 1-handlebody.
Since $ (\Sigma_h - Int(N_i ; \Sigma_h)$ is a punctured genus $h$ surface, this can be represented by a disk and $2h$ 1-handles attached to it. Since $B_i$ is a disk, $(B_i \times (\Sigma_h - Int(N_i ; \Sigma_h))$ is a 1-handlebody diffeomorphic to $\natural^{2h} S^1\times B_3$. 
The intersection $(B_i \times (\Sigma_h - Int(N_i ; \Sigma_h))$ and  $(B_{i+1}\times N_{i+1})$ is a $(B_i\cap B_{i+1})\times N_i$.
This is a disjoint $2g+1$ 3-balls in their boundaries. Since $(B_{i+1}\times N_{i+1})$ is a 4-ball, $X_i$ is a 1-handlebody of genus $2g+2h$.

In \cite{W2}, Williams showed that $X=X_1\cup X_2\cup X_2$ is a trisection.

\begin{theorem}[Theorem 3.3 in \cite{W2}] 
	$X=X_1\cup X_2\cup X_2$ is a genus $(2g+1)(2h+1)+1$ trisection.
\end{theorem}

The direction of this paper involves examining the above trisection in more detail to demonstrate how it will become a Weinstein trisection. To show this, we consider the Weinstein structure of $\Sigma_g - int (B_i)$ and $\Sigma_h - int (N_i)$. This is constructed in Section 5. 
It is easy to show that each of $N_i$ is mutually ambient isotopic since they are disjoint disks. Also, we can show that each of $B_i$ is mutually ambient isotopic. Actually, in Figure \ref{surfdeco}, $B_i$ is sent to $B_{i+1}$ by ambient isotopy along the diagonal from bottom-left to top-right. We can extend this ambient isotopy to an arbitrary genus since we can construct a decomposition of $\Sigma_g=B_1\cup B_2\cup B_2$ by connecting summing a torus by taking the regular neighborhood of a vertex depicted in Figure \ref{surfdeco}.
We consider this feature with a symplectic structure in the next subsection.


\subsection{symplectomorphisms compatible to the trisection.}
In this subsection, we see the self-symplectomorphisms of symplectic surfaces.
First of all, we provide a definition of symplectic and Hamiltonian isotopy. 
Let $(X, \omega)$ be a closed symplectic manifold and $H: X\to \mathbb{R}$ a smooth function on $X$.
Then there is a unique vector field $X_H$ such that 
\[
	\iota_{X_H}\omega = \omega(X_H, \cdot \ )=dH
\]
since $\omega$ is non-degenerate. We call $X_H$ a {\it Hamiltonian vector field} associated to the {\it Hamlltonian function} $H$.

\begin{definition}
	Let $\varphi_t$ be an ambient isotopy on $X$. 
	We say $\varphi_t$ is {\it symplectic isotopy} if $\varphi_t$ is symplectomorphism for every $t\in [0, 1]$.
	
	A symplectic isotopy is called a {\it Hamiltonian isotopy} if $\iota_{X_t}\omega$ is exact 1-form (i.e. $X_t$ is Hamiltonian vector field for every $t$) where $X_t$ is a vecotor field such that 
	\[
		\frac{d}{dt}\varphi_t = X_t\circ \varphi_t.
	\]
\end{definition}

By definition, a Hamiltonian isotopy is a symplectic isotopy. 

In Section 5, we use the following lemmas to construct Weinstein structures of $\Sigma_g - int (B_i)$ and $\Sigma_h - int (N_i)$.

To show Lemma \ref{lemma1}, we use the following lemma.
\begin{lemma}[cf. \cite{MS}, P.113]\label{hamiso}
Let $(\Sigma_g, \omega)$ be a symplectic closed surface and $B(r)^2$ a standard disk with radius $r$ for some $r\in \mathbb{R}$.
\[	\begin{array}{rccc}
			[0, 1]\  \times&B(r)^2  &\longrightarrow&   \Sigma_g \\
			\rotatebox{90}{$\in$}&& & \rotatebox{90}{$\in$} \\
			 (t, z)& &\longmapsto & \psi_t(z)
	\end{array}
\]
	be a smooth map such that $\psi_t: B(r)^2\rightarrow \Sigma_g$ is a symplectic embedding for every $t$.
	Then there exists a Hamiltonian isotopy
\[	\begin{array}{rccc}
			[0, 1]\times& \Sigma_g & \longrightarrow &  \Sigma_g \\
			\rotatebox{90}{$\in$}& & & \rotatebox{90}{$\in$} \\
			 (t, p) & &\longmapsto & \phi_t(p)
	\end{array}
\]
	such that 
	\[
		\phi_0=id,\  \phi_t\circ\psi_0=\psi_t
	\]
	for all $t$.
\end{lemma}

\begin{lemma}\label{lemma1}
	Let $\Sigma_g=B_1\cup B_2\cup B_3$ be a decomposition as in Lemma \ref{decolem1} and $\omega$ a symplectic form of  $\Sigma_g$.
	Suppose that each area of $B_i$ is identical for $i=1, 2, 3$ with respect to $\omega$. 
	Then there exists a symplectomorphism $\varphi_i$ such that $\varphi_i(B_i)=B_{i+1}$.
\end{lemma}
\begin{proof}
	Let $\psi_t: \Sigma_g\rightarrow \Sigma_g$ be a ambient isotopy, such that 
	\[
		\psi_1(B_i)= B_{i+1}.
	\]
	Let $p_i$ be a point in the interior of $B_i$. Then,  $\psi(p_i)$ is a point in the interior of $\psi(B_i)$.
	Then we can construct a symplectic embedding $\varphi'_t (B(r_i(t)))\rightarrow \Sigma_g$ so that 
	\[
		\varphi'_t(B(r_i(t)))=\psi_t(B_i)
	\]
	where $r_i(t)\in \mathbb{R}$ is a smooth function on $[0, 1]$ for $i=1, 2, 3$.
	Then, we perturb $\varphi'$ so that $r_i(t)=r_j$ for $i\neq j$ and any $t\in [0, 1]$, we obtain the follwoing family of symplectic embeddings $\varphi'_t$ such that
	\[
		\varphi'_t: B(r)\rightarrow \Sigma_g
	\]
	\[
		\varphi'_0(B(r))=B_i, \varphi'_1(B_i)=B_{i+1}.
	\]
	
	By Lemma \ref{hamiso}, we obtain a Hamiltonian isotopy $\varphi_t: \Sigma_g\rightarrow \Sigma_g$ such that 
	\[
		\varphi_0=id,\  \varphi_t\circ \varphi'_0=\varphi'_t.
	\]
	Then $\varphi_1$ is a symplectomorphism we want.
\end{proof}

From the result below, We can also assume that the $N_i$ in $\Sigma_h$ are sent to each other by symplectomorphism.
\begin{lemma}[Theorem A in \cite{B}]\label{hlem}
	Let $(\Sigma_h, \omega)$ be a symplectic closed surface and $\{p_1, \ldots, p_n\}$ and $\{q_1, \ldots, q_n\}$ are two sets of distinct points of $\Sigma_h$.
	Then there is a symplectomorphism $\varphi$ such that $\varphi(p_i)=q_i$ for $i=1, \ldots, n$ that is isotopic to identity by an isotopy which preserves the structure and leaves fixed every point of $\Sigma_h$ outside a compact set of arbitrarily small volume.
\end{lemma} 

\section{Stein and Weinstein structure and Riemann surface.}

\subsection{Stein structure of a Riemann surface}
We will review the concepts of Stein domains and that of Riemann surfaces.
A Riemann surface is a 2-manifold with a complex structure $J$.
A Stein manifold is a complex manifold that is embedded in $\mathbb{C}^N$, where $N$ is a natural number. Grauert provided a characterization of Stein manifolds based on the function they admit (refer to \cite{GR}). This function is known as a strictly plurisubharmonic or $J$-convex function.
In the context of a smooth function $f: M\rightarrow \mathbb{R}$, we say that it is exhausting if it is both proper and bounded from below.

\begin{definition}
	Let $(M, J)$ be a complex manifold and $f:M\to \mathbb{R}$ a function.
	$f$ is called a {\it plurisubharmonic or $J$-convex} function if the 2-form
	\[
		\omega_f:=-d(df\circ J)=-dd^{\mathbb{C}} f
	\]
	satisfies
	\[
		\omega_f(v, Jv)>0
	\]
	for every non-zero tangent vector $v\in TM$.
\end{definition}

It is well-known that a Riemann surface is a Stein manifold if and only if it is a non-compact.
Hence, $\Sigma^n_g$ becomes a Stein domain where $\Sigma^n_g$ is genus $g$ closed surface removed $n$ disks.
Hence we obtain the following:
\begin{lemma}\label{Stein_1}
	Let $\Sigma_g$ and $\Sigma_h$ be surfaces with genera $g$ and $h$ respectively, and $B_i$ and $N_i$ for $i=1, 2, 3$ are disks in $\Sigma_g$ and $\Sigma_h$ respectively that are defined in Section 3.
	Then, $\Sigma_h-int (N_i)$ and $\Sigma_g- B_{i+2}$ will be Stein domains for $i=1, 2, 3$. 
\end{lemma}


\subsection{Weinstein structure of Riemann surface}
Liouville vector field $\xi$ on symplectic manifold $(X, \omega)$ is a vector field such that $d (i_X \omega)=\omega$. 
If the Lie derivative of $\omega$ along $X$ is $\omega$, then $X$ is a Liouville vector field. This follows from $\omega$ is closed. Liouville domain is the symplectic manifold with some compatible vector fields with contact boundaries. 
\begin{definition}
	Let $(W, \omega)$ be a compact symplectic manifold with no empty boundary.
	$(W, \omega)$ is called a {\it Liouville domain} if there is a Liouville vector field defined globally and it is transversally out of the boundary. We denote it by $(W, \omega, X)$
\end{definition}

If the Liouville domain has a "compatible" Morse function, it is called a Weinstein domain.

\begin{definition}
	Let  $(W, \omega, X)$ be a Liouville domain.
	 $(W, \omega, X)$ is called a {\it Weinstein domain} if there exists a Morse function $f$ it is locally constant in $\partial W$ and $X$ is gradient-like for $f$.
\end{definition}

For a given symplectic manifold, it is difficult to determine whether it admits a Weinstein structure or not.
Also, generally,  it is difficult to construct a Weinstein structure but we sometimes construct it from a Stein structure.

We say exhausting J-convex function $f$ is {\it completely exhausting} if its gradient vector fields $\nabla_{f}f$ is complete where $\nabla_{f}f$ is a gradient respect to a Riemann metric $\omega_{f}(\cdot, J\cdot)$ for $\omega_f=-dd^{\mathbb{C}} f$.  The following is a well-known theorem:

\begin{theorem}[\cite{EG}]\label{StoW}
	Let $(V, J)$ be a Stein manifold and $f: V\rightarrow \mathbb{R}$ a completely exhausting $J$-convex Morse function. Then, 
	\[ 
		(\omega_{f}:=-dd^{\mathbb{C}}f, X_{f}:=\nabla_{f}f, f)
	\] 
is a Weinstein structure on $V$.
\end{theorem}

Hence, we can construct a Weinstein structure of $\Sigma_h-int (N_i)$ from a Morse $J$-convex function $f_i$ for $i=1, 2, 3$. 

\section{Proof of the main theorem}
Let us consider a 1-handlebody 
\[ 
	X_i=(\Sigma_h - int(N_i))\times B_i\cup N_{i+1}\times B_{i+1}
\]
 that constructs a trisection prescribed in Section 3. We show that each of $X_i$ admits a Weinstein structure with respect to a symplectic structure defined by a product structure of $\Sigma_g\times \Sigma_h$.


Let $g_1: \Sigma_g-B_3\rightarrow \mathbb{R}$ be a $J$-convex function such that $B_1$ contains only one critical point and its index is $0$.
Then we define the symplectic structure $\omega_g$ on $\Sigma_g$ so that 
\[
	\omega_g |_{\Sigma_g-B_3}=- dd^{\mathbb{C}} g_1.
\]
By Lemma \ref{lemma1}, there is a symplectomorphism $\varphi$ such that $\varphi(B_i)=B_{i+1}$. Then we define the $J$-convex function $g_2$ and $g_3$ as follows:
\[
		g_2=g_1\circ \varphi^{-1} |_{\Sigma_g - B_1}: \Sigma_g - B_1 \rightarrow \mathbb{R}, 
\]
\[
		g_3=g_2\circ \varphi^{-1}i |_{\Sigma_g - B_2}: \Sigma_g - B_2 \rightarrow \mathbb{R}.
\]
Then we can define the Weinstein structures on $\Sigma_g - B_{i+2}$ for $i=1, 2, 3$ by Theorem \ref{StoW}.

Let $N_1$ be a sufficiently small disk in $\Sigma_h$ and $f_1: \Sigma_h - N_1\rightarrow \mathbb{R}$ a $J$-convex function. Also, we suppose that $N_2$ is a sufficiently small regular neighborhood of index $0$ critical point of $f_1$ and $N_3$ a disk in the interior of $\Sigma_h - (N_1\cup N_2)$.
Then we define the symplectic structure $\omega_h$ of $\Sigma_h$ so that 
\[
	\omega_h|_{ \Sigma_h - N_1}= - dd^{\mathbb{C}} f_1.
\]
By Lemma \ref{hlem}, we can obtain the symplectomorphism $\varphi_2: (\Sigma, \omega_h)\rightarrow (\Sigma_h, \omega_h)$ such that 
\[
	\varphi(N_1)=N_2, \ \varphi(N_2)=N_3.
\]
Then, we obtain the $J$ convex function $f_2: \Sigma - N_2\rightarrow \mathbb{R}$ as follows:
\[
	f_2=f_1\circ \varphi_2^{-1}|_{\Sigma_h - N_2}: \Sigma_h - N_2 \rightarrow \mathbb{R}.
\]
Also, we can obtain a symplectomorphism  $\varphi_3: (\Sigma, \omega_h)\rightarrow (\Sigma_h, \omega_h)$ such that 
\[
	\varphi(N_2)=N_3, \ \varphi(N_3)=N_1.
\]
by Lemma \ref{hlem}. Then, we can define the $J$-convex function $f_3: \Sigma - N_3\rightarrow \mathbb{R}$ as follows:
\[
	f_3=f_2\circ \varphi_3^{-1}|_{\Sigma_h - N_3}: \Sigma_h - N_3 \rightarrow \mathbb{R}.
\]
Finally, we can construct a Weinstein structure on $\Sigma_h - N_i$ for $i=1, 2, 3$ by Theorem \ref{StoW}.


\begin{proof}[Proof of Theorem \ref{thm1}]
	Let $(\Sigma_h \times \Sigma_g, \omega)$ be a trivial bundle over surface with symplectic structure $\omega=pr_1^{\ast} \omega_h+pr_2^{\ast}\omega_g$ that is defined above.
	Now, the function $f_i+g_i$  is a $J$-convex Morse function of $(\Sigma_h- N_i)\times (\Sigma_g - B_{i+2})$.
	So, we have to show that the gradient vector field of this Morse function is outward and it is a Liouville vector field on $X_i$ with respect to a symplectic structure $\omega$. Let $\xi_i$ be the gradient vector field of $f_i+g_i$.
	
	First, we show that $\xi_i$ is a Liouville vector field on $X_i$ with respect to $\omega$.
	Now, $f_i$ and $g_i$ are $J$-convex Morse function, and $\omega_h|\Sigma_h- N_i$ and $\omega_g|\Sigma_g - B_{i+2}$ are symplectic structure defined by $-dd^{\mathbb{C}}f_i$ and $-dd^{\mathbb{C}}g_i$ respectively.
	Hence, by Theorem \ref{StoW}, $\xi_i$ gives a Liouville vector fields on $(\Sigma_h- N_i)\times (\Sigma_g - B_{i+2})$, particularly on $X_i$ for $\omega$.
	
	Next, we show that $\xi_i$ is transverse outwardly on the boudary of $X_i$.
	we recall that 
	\[
		X_i=(\Sigma_h - int(N_i))\times B_i\cup N_{i+1}\times B_{i+1}.
	\]
	Since $(\Sigma_h - int(N_i))\times int(B_i)$ and $N_{i+1}\times B_{i+1}$ intersects their boundaries, $\partial X_i$ can be desicribed as following:
	\begin{align*}
		\partial X_i =&( (\partial(\Sigma_h- int(N_i))\times B_i)\cup  ((\Sigma_h- int(N_i)\times \partial B_i) )) \\
		 		   &\cup ((\partial N_{i+1}\times B_{i+1} )\cup (N_{i+1}\cup \partial B_{i+1})) \\
		                    & - ( ((\Sigma_h - int(N_i))\times int(B_i))\cap (N_{i+1}\times B_{i+1} ))
	\end{align*}
	We note that  $((\Sigma_h - int(N_i))\times int(B_i))\cap (N_{i+1}\times B_{i+1} )= N_{i+1}\times (B_i\cap B_{i+1})$.
	Then, we will see whether $\xi_i$ is outward for each region.
	We note that $B_i$ is a neighborhood of an index $0$ critical point of $g_i$ by definition of $g_i$. 
	Then, the restriction of $\xi_i$ to $B_i$ is transverse outwardly on its boundary.
	Since the restriction of $\xi_i$ to $B_i$ and $\Sigma_h-int (N_i)$ is transverse outwardly on its boundary, $\xi_i$ transverse outwardly on $(\partial(\Sigma_h- int(N_i))\times B_i)$ and $((\Sigma_h- int(N_i)\times \partial B_i)$.
	
	We note that $N_{i+1}$ is a neighborhood of an index $0$ critical point of $f_i$ by definition of $f_i$.
	$\xi_i$ transverse outwardly on $\partial (N_{i+1}\times B_{i+1})$ except $N_{i+1}\times (B_{i} \cap \partial B_{i+1})$ since the restriction of $\xi$ to $B_{i+1}$ transeverse outwardly on $\partial B_{i+1}$ except $B_{i}\cap B_{i+1}$.
	But we see that 
	\[
		((\Sigma_h - int(N_i))\times int(B_i))\cap ( N_{i+1}\times B_{i+1}) = N_{i+1}\times (B_i\cap B_{i+1}).
	\]
	Hence the region $N_{i+1}\times (B_i\cap \partial B_{i+1})$  is not included a boundary of $X_i$.
	Hence $\xi_i$ transverse outwardly on $\partial X_i$ entirely.

\end{proof}

\section{The case where $g=h=0$.}
In this section, we give an example of the Weinstein trisection of $S^2\times S^2$.
It is a trivial $S^2$ bundle over $S^2$.
We denote it $S_1\times S_2$.
Also, $S_2$ has a decomposition described above. More precisely, we assume that $S_2$ is decomposed into three disks $B_1$, $B_2$, and $B_3$ as in Figure \ref{surfdeco}.
We note that the Weinstein trisection of $S^2\times S^2$ is constructed in former articles \cite{LMS}.
But in this section, we will describe it slightly more explicitly.

We construct a Weinstein trisection of $S^2\times S^2$ as the following steps:
\begin{enumerate}
	\item[Step 1:] For a given symplectic structure on $S_1 - N_i$, we define a Morse function $f_i$ so that $N_{i+1}$ is a neighborhood of index $0$ critical point of $f_i$, it does not contain the other critical point and whose gradient flow is a Liouville vector field of the symplectic structure. Furthermore, each of $N_i$ is mutually symplectically isotopic to each other.
	
	\item[Step 2:] For a given symplectic structure on $S_2 - B_i$, we define a Morse function $g_i$ so that $B_i$ is a neighborhood of index $0$ critical point of $g$ and does not contain the other critical points of $g$, $B_{i+1}$ is a neighborhood of index $2$ critical point of $g$ and does not contain the other critical points of $g$ and whose gradient vector field is a Liouville vector field for the symplectic structure. Furthermore, each of $B_i$ is mutually symplectically isotopic to each other.
	
	\item[Step 3:] $f+g$ is a Morse function on $X_i$ and a gradient vector field of it is a Liouville vector field for the products of the symplectic structure of $S_1$ and $S_2$ and $(X_i, \omega, g+f, grad(g+f))$ will be a Weinstein domain.
	
\end{enumerate}

To begin the steps above, we review the K\"{a}hler structure on $S^2$. First of all, we review the Fubini-study form of $\mathbb{C}P^1$.

Let $(z_1, z_2)$ be a homogenious coordinate of $\mathbb{C}P^1$. Then we can take a chart of $\mathbb{C}P^1$ by taking $\phi_i:\mathbb{C}P^1\to \mathbb{C}$ for  $z_j\neq 0$
\[
	\phi_i(z_1, z_2)=\frac{z_i}{z_j} \ (i\neq j).
\]
We denote this chart by $\{U_i, \phi_i\}_{i=1, 2}$.
Then we define the Fubini-Study form by 
\[
	\omega_{FS} = 2 \frac{\partial^2 \log(1+|z|^2 )}{\partial z\partial \bar{z}} dz \wedge d\bar{z}
\]
where $z=z_1/z_2$.
It is known that this form gives us the  K\"{a}hler form on $U_1$, and in this case, $f= \log(1+|z|^2)$ gives a  K\"{a}hler potential.
The critical point of $f$ is only $z=0$. This means that $z_1=0$ since $z_2\neq 0$.
Also, we note that this is a $J$-holomorphic function on $U_1$. 
Hence, $U_1$ is a Stein manifold with $J$-holomorphic function $f$.

Next, we review the identification between $S^2$ and $\mathbb{C}P^1$. 
We denote the polar coordinate of the unit sphere by the following (See Figure \ref{stproj}):
\begin{align*}
	x=&\sin \theta_1 \cos\theta_2, \\
	y=&sin\theta_1 \sin \theta_2, \\
	z=&\cos \theta_1.
\end{align*}
	 \begin{figure}[h]
		\centering
		\includegraphics[scale=0.6]{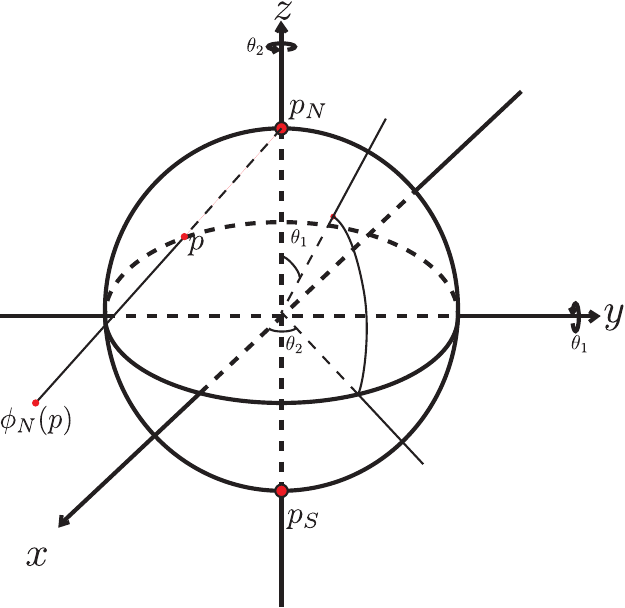}
		\caption{This Figure represents a polar coordinate of $S^2$. In this figure, $\phi_N$ is a stereographic projection.}
		\label{stproj}
	\end{figure}

Then we can represent a point on $S^2$ by $(x, y, z)\in S^2\subset \mathbb{R}^3$. We denote the stereographic projection $\phi_N: S^2\backslash p_N\to \mathbb{R}^2$ where $p_N=(0, 0, 1)$.
Then we can define the diffeomorphism $\Phi : \mathbb{C}P^1\to S^2$ by 
\begin{equation*}
\Phi(z_1, z_2)=
\begin{cases}{}
\phi_N(z)&	(z_2\neq 0)\\
p_N&		(z_2=0)
\end{cases}     
\end{equation*}        
where $z=z_1/z_2$.
We sometimes use this diffeomorphism to define the function on $S^2$ below.

We identify $f\circ \Phi^{-1}$ and ${\Phi^{-1}}^\ast \omega_{FS}$ by $f$ and $\omega_{FS}$ respectively.

\subsection{Step 1: Weinstein structure on $S_1- N_i$}
	We define the region $N_1$, $N_2$ and $N_3$ on $S_1$.
	Let $\varphi$ be a $4/3 \pi$-rotation along a $y$-axis.
	Then we denote $P_N=p_1$, $\varphi(P_N)=p_2$ and $\varphi^2(P_N)=p_3$.
	Also, we denote $P_S=q_1$, $\varphi(P_S)=q_2$ and $\varphi^2(P_S)=q_3$.
	Then we define $N_1$ as a neighborhood of $p_1$ so that it contains $q_3$ and does not contain $p_i$ and $q_j$ for $i=2, 3, j=1, 2$.
	$N_2$ and $N_3$ are defined by $\varphi(N_1)$ and $\varphi^2(N_3)$ respectively (See Figure \ref{S_1}).
	
	 \begin{figure}[h]
		\centering
		\includegraphics[scale=0.6]{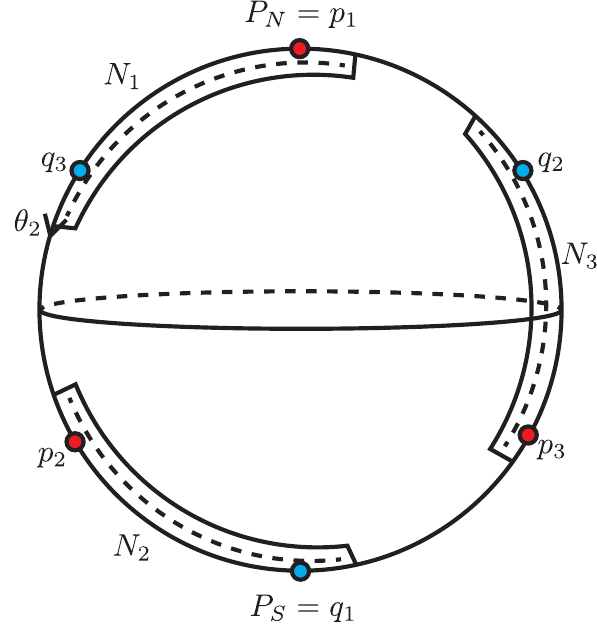}
		\caption{Each of $N_i$ is a region as in this figure. They are mutually permuted by rotation $\varphi$}
		\label{S_1}
	\end{figure}
	
	We take rotation $\varphi$ along a $y-axis$ such that $\varphi(N_1)\cap N_1$ is an empty set.
	Also, the we define $\varphi(N_1)=N_2$ and $\varphi(N_2)=N_3$.
	Since a rotation is volume-preserving, this is a sympelctomorphism that sends $N_i$ to each other.
	Then $f$ define the $J$-convex function on $S^2-N_1$ respect to $\omega_{FS}$.
	By using Theorem \ref{StoW}, we can obtain the Weinstein structure on $S_1- N_1$.
	By permuting $N_i$ by rotation, we can also obtain the Weinstein structure on $S_1 - N_i$ for $i=2, 3$.
\subsection{Step 2: Weinstein structure on $S_2 - B_{i+2}$}
	We define the region $B_i$ in $S_2$ for $i=1, 2, 3$ as in Lemma \ref{surfdeco}.
	In this case, the intersection $B_i\cap B_j$ for $i\neq j$ is an arc.
	We define the region $B_i$ as follows:
	\[
		B_1=\left\{(\sin \theta_1 \cos\theta_2, \sin\theta_1 \sin \theta_2, \cos \theta_1)\in S^2\mid \frac{4}{12}\pi \leq \theta_2 \leq \frac{12}{12}\pi \right \}, 
	\]
	\[
		B_2 = \left\{(\sin \theta_1 \cos\theta_2, \sin\theta_1 \sin \theta_2, \cos \theta_1)\in S^2 \mid \frac{12}{12}\pi \leq \theta_2 \leq \frac{20}{12}\pi \right \},
	\]
	\[
		B_2 = \left\{(\sin \theta_1 \cos\theta_2, \sin\theta_1 \sin \theta_2, \cos \theta_1)\in S^2 \mid \frac{20}{12}\pi \leq \theta_2 \leq \frac{24}{12}\pi \right \}.
	\]
	We can see each region as in Figure \ref{Bregion}.
	 \begin{figure}[h]
		\centering
		\includegraphics[scale=0.7]{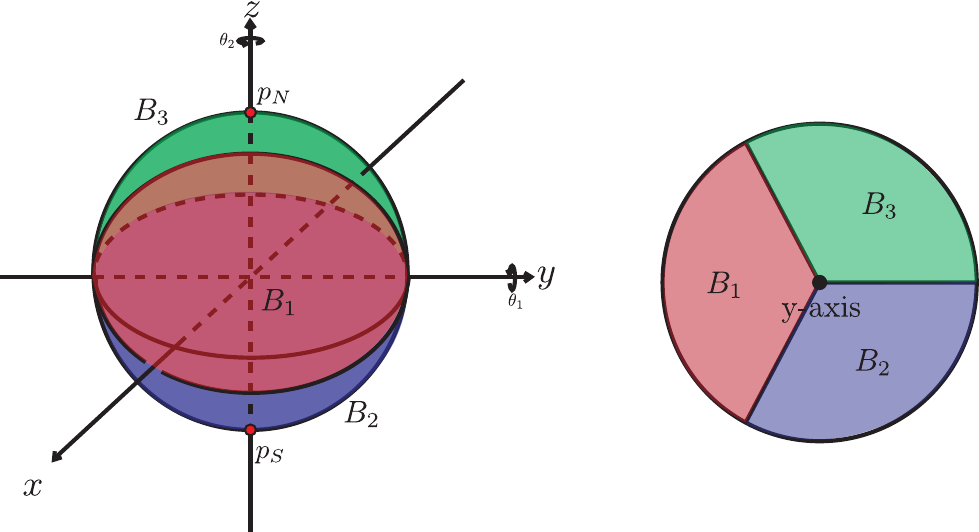}
		\caption{The Left Figure describes the regions $B_i$ and the right figure is a picture view from a direction perpendicular to the $xz$ plane.}
		\label{Bregion}
	\end{figure}
 	It is easy to see that each of $B_i$ is mutually permuted by rotating along a $y$-axis.
	Also, we can define the $J$-holomorphic function $g_1$ respect to $\omega_{FS}$ on $S_2 - int (B_3)$ from the function $f$ defined above.
	Also, composing a $g_1$ and rotating, we can also obtain the $J$-convex function on $S_2 - int (B_i)$ for $i=1, 2$.
	We will check that the critical point of $g_1$ is included in $B_2$. 
	
	$g_1$ is a restriction of $f$ to $S^2- int(B_3)$ it was defined before the steps. 
	Its critical point is $z_1=0$.
	This is corresponds to a point $p_S=(0, 0, -1)$ by the identification map $\Phi$.
	Also, this critical point has an index $0$ since by Poincare-Hopf theorem,  $f$ is a Morse function on disk and it has only one critical point.
	Hence we can assume $B_2$ is a neighborhood of index $0$ critical point of $g_1$.
\subsection{Step 3: $(X_i, \omega\mid_{X_i}, (g+f)\mid_{X_i}, grad(g+f))$ is a Weinstein domain.}
	We remain the following: First, we have to check the gradient flow of $g_i + f_i$ is the Liouville vector field. Next, we have to check it is outward.
	By Theorem \ref{StoW}, we can see that the gradient-like flow of $f_i$ and $g_i$ are Liouville vector fields with respect to $\omega_{FS}$.
	Hence we can show that the gradient-like flow of $f_i + g_i$ is a Liouville vector field with respect to $\omega_{FS}+\omega_{FS}=2\omega_{FS}$.
	
	Finally, we will check it is outward. Following the construction of trisection above, $X_i$ is a set as follows:
	\[
		X_i = ((S_1 - N_i)\times B_{i} )\cup (N_{i+1}\times B_{i+1})
	\]
	We show the case where $i=1$. 
	We note that each of $(S_1 - N_1) \times B_1$ and $N_{2}\times B_{2}$ are diffeomorphic to a 4-ball since $S_1 - N_1$ is a disk.
	We defined the vector field on $S_{1} - N_1$ and $B_1\cup B_2$ by the gradient-like vector field of $f_1$ and $g_1$ respectively. They are outward on each of $S_{1} - N_1$ and $B_1\cup B_2$ respectively.
	We denote an arc $B_1\cap B_2$ $\alpha$.
	Then the intersection of two 4-ball $(S_1 - N_1) \times B_1$ and $N_{2}\times B_{2}$ is a 3-ball $N_2\times \alpha$.
	We shall consider the boundary of $X_i$. 
	$\partial X_i$ is a union of  $\{(S_1 - N_1) \times B_1 \} - (N_2\times \alpha)$ and $N_{2}\times B_{2} - N_2\times \alpha$.
	Now, the gradient-like vector field of $f_1 + g_1$ on $(S_1 - N_1) \times B_1$ outward since $(S_1 - N_1)$ and $B_1$ are neiborhood of index $0$ critical point of $f_1$ and $g_1$ respectively.
	Also, the gradient-like vector field of $f_1 + g_1$ on $N_{2}\times B_{2}$ is outward except $N_2\times \alpha$ since we can assume that $N_2$ is a neighborhood of index $0$ critical point of $f_1$ and the gradient-like vector field of $g_1$ is outward on $B_1\cup B_2$.
	This is enough to show that $f_1 + g_1$ is outward on $X_1$.
	
	\section{Acknowledgement}
The author thanks Takahiro Oba for a very meaningful discussion and suggestion about research in symplectic topology, also giving comments on a draft of this paper.

\end{document}